\documentclass[reqno, 11pt]{amsart}
\usepackage{amsthm, amsmath, amssymb, graphicx, tikz,enumerate,paralist,bbm,setspace,bbm}
\usepackage[left=1in,right=1in,top=1in,bottom=1in]{geometry}

\usepackage[color,final]{showkeys} 
\usepackage[colorlinks=false, pdfstartview=FitV, 
pagebackref=false,hidelinks]{hyperref}

\definecolor{refkey}{gray}{.75}
\definecolor{labelkey}{gray}{.5}

\usetikzlibrary{calc,intersections,through,backgrounds,shapes,patterns}
\tikzstyle{p}+=[fill=black, circle, minimum width = 1pt, inner sep =
1pt]

\newtheorem{theorem}{Theorem}

\newtheorem{lemma}[theorem]{Lemma}

\theoremstyle{definition}

\theoremstyle{remark}

\newtheorem*{remarks}{Remarks}

\newcommand{\abs}[1]{\left\lvert#1\right\rvert}

\newcommand{\floor}[1]{\left\lfloor #1 \right\rfloor}

\newcommand{\paren}[1]{\left( #1 \right)}

\newcommand{\x}{\times}

\newcommand{\ZZ}{\mathbb{Z}}

\newcommand{\cP}{\mathcal{P}}

\newcommand{\tN}{\widetilde{N}}
\newcommand{\tA}{\widetilde{A}}
\newcommand{\tL}{\widetilde{\Lambda}}

\title[A short
  proof of the multidimensional Szemer\'edi theorem in the primes]{A short
  proof of the \\ multidimensional Szemer\'edi theorem in the primes}

\author{Jacob Fox}
\address{Department of Mathematics\\
MIT\\
Cambridge\\
MA 02139-4307}
\email{fox@math.mit.edu}

\author{Yufei Zhao}
\address{Department of Mathematics\\
MIT\\
Cambridge\\
MA 02139-4307}
\email{yufeiz@math.mit.edu}

\thanks{The first author was supported by a Packard Fellowship, NSF CAREER award  DMS-1352121, by an Alfred P. Sloan
Fellowship, and by an MIT NEC Corporation Fund Award, and the second author was
  supported by an internship at Microsoft Research New England and a Microsoft Research PhD Fellowship.}

\begin{document}

\begin{abstract}
Tao conjectured that every dense subset of $\mathcal{P}^d$, the $d$-tuples of primes, contains
constellations of any given shape.  This was very recently proved by Cook,
Magyar, and Titichetrakun and independently by Tao and Ziegler. Here we give
a simple proof using the Green-Tao theorem on linear
equations in primes and the Furstenberg-Katznelson multidimensional
Szemer\'edi theorem.
\end{abstract}

\maketitle

Let $\cP_N$ denote the set of primes at most $N$, and let $[N] := \{1, 2, \dots, N\}$. Tao \cite{Tao06jam} conjectured the following result as a natural extension of the Green-Tao theorem \cite{GT08} on arithmetic progressions in the primes and the Furstenberg-Katznelson \cite{FK78} multidimensional generalization of Szemer\'edi's theorem. Special cases of this conjecture were proven in \cite{CM12} and \cite{MT13ar}. The conjecture was very recently resolved by Cook, Magyar, and Titichetrakun~\cite{CMT13ar} and independently by Tao and Ziegler~\cite{TZ13ar}.

\begin{theorem} \label{thm:rel-sz-primes}
Let $d$ be a positive integer, $v_1, \dots, v_k \in \ZZ^d$, and $\delta > 0$. Then, if $N$ is sufficiently large, every subset $A$ of $\cP_N^d$ of cardinality $\abs{A} \geq \delta \abs{\cP_N}^d$ contains a set of the form $a + tv_1, \dots, a + tv_k$, where $a \in \ZZ^d$ and $t$ is a positive integer.
\end{theorem}

In this note we give a short alternative proof of the theorem, using
the landmark result of Green and Tao \cite{GT10} (which is conditional
on results later proved in \cite{GT12} and with Ziegler in \cite{GTZ12}) on the asymptotics for the number of primes satisfying certain systems of linear equations, as well as the following multidimensional generalization of Szemer\'edi's theorem established by Furstenberg and Katznelson \cite{FK78}.

\begin{theorem}[Multidimensional Szemer\'edi theorem \cite{FK78}] \label{thm:FK}
Let $d$ be a positive integer, $v_1, \dots, v_k \in \ZZ^d$, and $\delta > 0$. If $N$ is sufficiently large, then every subset $A$ of $[N]^d$ of cardinality $\abs{A} \geq \delta N^d$ contains a set of the form $a + tv_1, \dots, a + tv_k$, where $a \in \ZZ^d$ and $t$ is a positive integer.
\end{theorem}

To prove Theorem~\ref{thm:rel-sz-primes}, we begin by fixing  $d, v_1, \dots, v_k, \delta$.
Using Theorem~\ref{thm:FK}, we can fix a large integer $m > 2d/\delta$ so that any subset of $[m]^d$ with at least $\delta m^d / 2$ elements contains a set of the form $a + tv_1, \dots, a + tv_k$, where $a \in \ZZ^d$ and $t$ is a positive integer.

We next discuss a sketch of the proof idea. The Green-Tao theorem \cite{GT08} (also see \cite{CFZrelsz,CFZgt} for some recent simplifications) states that there are arbitrarily long arithmetic progressions in the primes. It follows that for $N$ large, $\cP_N^d$ contains homothetic copies of $[m]^d$. We use a Varnavides-type argument \cite{Va} and consider a random homothetic copy of the grid $[m]^d$ inside $\cP_N^d$. In expectation, the set $A$ should occupy at least a $\delta/2$ fraction of the random homothetic copy of $[m]^d$. This follows from a linearity of expectation argument. Indeed, the Green-Tao-Ziegler result \cite{GT10,GT12,GTZ12} and a second moment argument imply that most points of $\cP_N^d$ appear in about the expected number of such copies of the grid $[m]^d$. Once we find a  homothetic copy of $[m]^d$ containing at least $\delta m^d /2$ elements of $A$, we obtain by Theorem~\ref{thm:FK} a subset of $A$ of the form $a + tv_1, \dots, a + tv_k$, as desired.

To make the above idea actually work, we first apply the $W$-trick as
described below. This is done to avoid certain biases in the primes. We also only consider homothetic
copies of $[m]^d$ with common difference $r \leq N/m^2$ in order to
guarantee that almost all elements of $\cP_N^d$ are in about the same
number of such homothetic copies of $[m]^d$.

\begin{remarks}
This argument also produces a relative multidimensional Szemer\'edi theorem,
where the complexity of the linear forms condition on the majorizing
measure depends on $d, v_1, \dots, v_k$ and $\delta$. It seems
plausible that the dependence on $\delta$ is unnecessary; this was
shown for the one-dimensional case in \cite{CFZrelsz}. Our arguments
share some features with those of Tao and Ziegler \cite{TZ13ar}, who also use the results in \cite{GT10,GT12,GTZ12}. However, the proof in \cite{TZ13ar} first establishes a relativized version of the Furstenberg correspondence principle and then proceeds in the ergodic theoretic setting, whereas we go directly to the multidimensional Szemer\'edi theorem. Cook, Magyar, and Titichetrakun
\cite{CMT13ar} take a different approach and develop a relative
hypergraph removal lemma from scratch, and they also require a linear
forms condition whose complexity depend on $\delta$.

Conditional on a certain polynomial extension of the Green-Tao-Ziegler
result (c.f.~the Bateman-Horn conjecture \cite{BT}), one can also
combine this sampling argument with the polynomial extension of
Szemer\'edi's theorem by Bergelson and Leibman \cite{BL} to obtain a
polynomial extension of Theorem~\ref{thm:rel-sz-primes}.
\end{remarks}

The hypothesis that $\abs{A} \geq \delta \abs{\cP_N}^d$ implies that
\begin{equation} \label{eq:density-L}
\sum_{n_1, \dots, n_d \in [N]} 1_A (n_1, \dots, n_d) \Lambda'(n_1) \cdots \Lambda'(n_d) \geq (\delta - o(1)) N^d,
\end{equation}
where $1_A$ is the indicator function of $A$, and $o(1)$ denotes some quantity that goes to zero as $N \to \infty$, and $\Lambda'(p) = \log p$ for prime $p$ and $\Lambda'(n) = 0$ for nonprime $n$.

Next we apply the $W$-trick \cite[\S5]{GT10}. Fix some slowly growing function $w = w(N)$; the choice $w:=\log\log\log N$ will do.  Define $W := \prod_{p \leq w} p $ to be the product of all primes at most $w$. For each $b \in [W]$ with $\gcd(b,W) = 1$, define
\[
\Lambda'_{b,W}(n) := \frac{\phi(W)}{W} \Lambda'(Wn + b)
\]
where $\phi(W) = \#\{b \in [W] : \gcd(b,W) = 1\}$ is the Euler totient function. Also define
\[
1_{A_{b_1, \dots, b_d,W}}(n_1, \dots, n_d) := 1_A(Wn_1 + b_1,\dots, Wn_d + b_d).
\]
By \eqref{eq:density-L} and the pigeonhole principle,  we can choose $b_1, \dots, b_d \in [W]$ all coprime to $W$ so that
\begin{equation}\label{eq:N/W}
\sum_{1 \leq n_1, \dots, n_d \leq N/W} 1_{A_{b_1, \dots, b_d,W}}(n_1, \dots, n_d) \Lambda'_{b_1,W}(n)\Lambda'_{b_2,W}(n)\cdots\Lambda'_{b_d,W}(n) \geq (\delta - o(1)) \paren{\frac{N}{W}}^d,
\end{equation}
We shall write
\[
\tN := \floor{N/W}, \qquad R:= \lfloor \tN/m^2 \rfloor, \qquad
\tA := 1_{A_{b_1,\dots,b_d,W}} \qquad \text{and} \qquad \tL_j := \Lambda'_{b_j,W}
\]
(all depending on $N$). So \eqref{eq:N/W} reads
\begin{equation} \label{eq:A-delta}
\sum_{n_1, \dots, n_d \in [\tN]} \tA(n_1,\dots,n_d) \tL_1(n_1)\tL_2(n_2)\cdots \tL_d(n_d) \geq (\delta - o(1)) \tN^d
\end{equation}

The Green-Tao result \cite{GT10} (along with \cite{GT12,GTZ12}) says that $\Lambda'_{b_j,W}$ acts
pseudorandomly with average value about $1$ in terms of counts of linear forms. The statement below is an easy corollary of \cite[Thm.~5.1]{GT10}.

\begin{theorem}[Pseudorandomness of the $W$-tricked primes] \label{thm:GT}
Fix a linear map $\Psi= (\psi_1, \dots,  \psi_t) : \ZZ^d \to \ZZ^t$ (in particular $\Psi(0) = 0$) where no two $\psi_i$, $\psi_j$ are linearly dependent. Let $K \subseteq[-\tN,\tN]^d$ be any convex body. Then, for any $b_1, \dots, b_t \in [W]$ all coprime to $W$, we have
\[
\sum_{n \in K \cap \ZZ^d} \prod_{j \in [t]} \Lambda'_{b_j, W}(\psi_j(n))
= \#\{n \in K \cap \ZZ^d: \psi_j(n) > 0 \  \forall j\} + o(\tN^d).
\]
where $o(\tN^d) := o(1) \tN^d$. Note that the error term does not depend on $b_1, \dots, b_t$ (although it does depend on $\Psi$).
\end{theorem}

The next lemma shows that $A$ in expectation contains a
considerable fraction of a random homothetic copy of $[m]^d$ with
common difference at most $R = \lfloor N/m^2 \rfloor$ in the $W$-tricked subgrid of $\cP_N^d$.

\begin{lemma}\label{lem:A-AP}
If $\tA$ satisfies \eqref{eq:A-delta}, then
\begin{multline}\label{eq:A-AP}
\sum_{\substack{n_1, \dots, n_d \in [\tN] \\
                r \in [R]}}
     \paren{\sum_{i_1, \dots, i_d \in [m]} \tA(n_1 + i_1 r, \dots, n_d
       + i_d r)}
     \prod_{j \in [d]} \prod_{i \in [m]} \tL_j(n_j + i r) \\
\geq
(\delta m^d - dm^{d-1} - o(1))R\tN^{d}.
\end{multline}
\end{lemma}

\begin{proof}[Proof of Theorem~\ref{thm:rel-sz-primes} (assuming Lemma \ref{lem:A-AP}) ]
By Theorem~\ref{thm:GT} we have
\[
\sum_{\substack{n_1, \dots, n_d \in [\tN] \\
                r \in [R]}}
\prod_{j \in [d]} \prod_{i \in [m]} \tL_j(n_j + i r) =
(1 + o(1)) R\tN^{d},
\]
So by \eqref{eq:A-AP}, for sufficiently large $N$, there exists some choice of $n_1, \dots, n_d \in [\tN]$ and $r \in [R]$ so that
\[
\sum_{i_1, \dots, i_d \in [m]} \tA(n_1 + i_1 r, \dots, n_d + i_d r) \geq
\frac12 \delta m^d.
\]
This means that a certain dilation of the grid $[m]^d$ contains at
least $\delta m^d/2$ elements of $A$, from which it follows by the
choice of $m$ that it must contain a set of the form $a + t v_1,
\dots, a + t v_k$.
\end{proof}

Lemma~\ref{lem:A-AP} follows from the next lemma by summing over all choices of  $i_1, \dots, i_d \in [m]$.

\begin{lemma} \label{lem:A-AP-j}
Suppose $\tA$ satisfies \eqref{eq:A-delta}. Fix $i_1, \dots, i_d \in [m]$. Then we have
\begin{equation}\label{eq:A-AP-j}
\sum_{\substack{n_1, \dots, n_d \in [\tN] \\
                r \in [R]}}\tA(n_1 + i_1 r, \dots, n_d + i_d r)
     \prod_{j \in [d]} \prod_{i \in [m]} \tL_j(n_j + i r)
\geq
\paren{\delta  - \frac{d}{m} - o(1)} R\tN^{d}.
\end{equation}
\end{lemma}

\begin{proof}
By a change of variables $n'_j = n_j + i_j r$ for each $j$, we write the LHS of \eqref{eq:A-AP-j} as
\begin{equation} \label{A-AP-j-preshift}
\sum_{r \in [R]}
\sum_{\substack{n'_1, \dots, n'_d \in \ZZ \\
                n'_j - i_j r \in [\tN] \ \forall j}}
 \tA(n'_1, \dots, n'_d) \prod_{j \in [d]} \prod_{i \in [m]} \tL_j(n'_j + (i-i_j) r).
\end{equation}
Recall that $R = \lfloor \tN / m^2 \rfloor$. Note that \eqref{A-AP-j-preshift} is at least
\begin{equation} \label{A-AP-j-postshift}
\sum_{r \in [R]} \
\sum_{\tN/m < n'_1, \dots, n'_d \leq \tN}
 \tA(n'_1, \dots, n'_d) \prod_{j \in [d]} \prod_{i \in [m]} \tL_j(n'_j + (i-i_j) r).
\end{equation}
By \eqref{eq:A-delta} and Theorem~\ref{thm:GT} we have
\begin{equation} \label{eq:A-delta-trunc}
\sum_{\tN/m < n_1, \dots, n_d \leq \tN} \tA(n_1,\dots,n_d) \tL_1(n_1)\tL_2(n_2)\cdots \tL_d(n_d) \geq \paren{\delta - \frac{d}{m} -  o(1)} \tN^d
\end{equation}
(the difference between the left-hand side sums of \eqref{eq:A-delta}
and \eqref{eq:A-delta-trunc} consists of terms with $(n_1, \dots,
n_d)$ in some box of the form $[\tN]^{j-1} \x [\tN/m] \x [\tN]^{d-j}$,
which can be upper bounded by using $\tA \leq 1$, applying
Theorem~\ref{thm:GT}, and then taking the union bound
over $j \in [d]$). It remains to show that
\[
\eqref{A-AP-j-postshift} - R\cdot (\text{LHS of (\ref{eq:A-delta-trunc})}) = o(\tN^{d+1}).
\]
We have
\begin{align*} \label{eq:A-AP-j-comp}
&\eqref{A-AP-j-postshift} - R\cdot (\text{LHS of (\ref{eq:A-delta-trunc})})
\\
&=
\sum_{\substack{\tN/m < n'_1, \dots, n'_d \leq \tN \\ r \in [R]}}
 \tA(n'_1, \dots, n'_d) \paren{ \prod_{j \in [d]} \prod_{i \in [m]} \tL_j(n'_j + (i-i_j) r) - \prod_{j \in [d]} \tL_j(n'_j)}
\\
&=
\sum_{\tN/m < n'_1, \dots, n'_d \leq \tN}
 \tA(n'_1, \dots, n'_d) \paren{\prod_{j \in [d]} \tL_j(n'_j)} \paren{\sum_{r \in [R]} \paren{\prod_{j \in [d]} \prod_{i \in [m] \setminus \{i_j\}} \tL_j(n'_j + (i-i_j) r) \ - 1}}.
\end{align*}
By the Cauchy-Schwarz inequality and $0 \leq \tA \leq 1$, the above
expression can be bounded in absolute value by $\sqrt{ST}$, where
\begin{align*}
S &= \sum_{\tN/m < n'_1, \dots, n'_d \leq \tN } \prod_{j \in [d]} \tL_j(n'_j),
\\
T &= \sum_{\tN/m < n'_1, \dots, n'_d \leq \tN }
 \paren{\prod_{j \in [d]} \tL_j(n'_j)}\paren{\sum_{r \in
     [R]} \paren{\prod_{j \in [d]} \prod_{i \in [m] \setminus
       \{i_j\}} \tL_j(n'_j + (i-i_j) r) \ - 1}}^2 \\
 &= T_1 - 2T_2 + T_3,
\end{align*}
and
\begin{align*}
  T_1 &= \sum_{\substack{\tN/m < n'_1, \dots, n'_d \leq \tN \\ r, r' \in [R]}}
 \prod_{j \in [d]} \tL_j(n'_j) \prod_{i \in [m] \setminus
       \{i_j\}} \tL_j(n'_j + (i-i_j) r) \tL_j(n'_j + (i-i_j) r'),
     \\
       T_2 &= \sum_{\substack{\tN/m < n'_1, \dots, n'_d \leq \tN \\ r, r' \in [R]}}
 \prod_{j \in [d]} \tL_j(n'_j) \prod_{i \in [m] \setminus
       \{i_j\}} \tL_j(n'_j + (i-i_j) r) ,
\\
       T_3 &= \sum_{\substack{\tN/m < n'_1, \dots, n'_d \leq \tN \\ r, r' \in [R]}}
 \prod_{j \in [d]} \tL_j(n'_j).
\end{align*}
By Theorem~\ref{thm:GT} we have $S = O(\tN^{d})$, and $T_1, T_2,
T_3$ pairwise differ by $o(\tN^{d+2})$, so that $T =
o(\tN^{d+2})$. Thus $\sqrt{ST} = o(\tN^{d+1})$, as desired.
\end{proof}

\vspace{0.1cm}
\noindent {\bf Acknowledgment.} \, We thank David Conlon, Ben Green,
and Terence Tao, and the anonymous referee for helpful comments.


\begin{thebibliography}{10}

\bibitem{BT}
P.~T.~Bateman and R.~A.~Horn,
\emph{A heuristic asymptotic formula concerning the distribution of prime numbers},
Math. Comp. \textbf{16} (1962), 363--367.

\bibitem{BL}
V.~Bergelson and A.~Leibman,
\emph{Polynomial extensions of van der Waerden's and Szemerédi's theorems}, J. Amer. Math. Soc. \textbf{9} (1996), no. 3, 725--753.

\bibitem{CFZrelsz}
D.~Conlon, J.~Fox, and Y.~Zhao, \emph{A relative {S}zemer\'edi theorem},
  arXiv:1305.5440.

\bibitem{CFZgt}
D.~Conlon, J.~Fox, and Y.~Zhao, \emph{The {G}reen--{T}ao theorem: an
  exposition}, arXiv:1403.2957.


\bibitem{CM12}
B.~Cook, A.~Magyar, \emph{Constellations in ${\mathbb P}^d$}, Int. Math. Res. Not. IMRN \textbf{2012} (2012), no.~12, 2794--2816.

\bibitem{CMT13ar}
B.~Cook, A.~Magyar, and T.~Titichetrakun, \emph{A multidimensional
  {S}zemer\'edi theorem in the primes}, arXiv:1306.3025.

\bibitem{FK78}
H.~Furstenberg and Y.~Katznelson, \emph{An ergodic {S}zemer{\'e}di theorem for
  commuting transformations}, J. Analyse Math. \textbf{34} (1978), 275--291.


\bibitem{GT08}
B.~Green and T.~Tao, \emph{The primes contain arbitrarily long arithmetic
  progressions}, Ann. of Math. \textbf{167} (2008), no.~2, 481--547.

\bibitem{GT10}
\bysame, \emph{Linear equations in primes}, Ann. of Math. \textbf{171} (2010),
  no.~3, 1753--1850.

\bibitem{GT12}
\bysame, \emph{The {M}\"obius function is strongly orthogonal to
  nilsequences}, Ann. of Math. \textbf{175} (2012), no.~2, 541--566.

\bibitem{GTZ12}
B.~Green, T.~Tao, and T.~Ziegler, \emph{An inverse theorem for the {G}owers
  {$U^{s+1}[N]$}-norm}, Ann. of Math. \textbf{176} (2012), no.~2,
  1231--1372.

\bibitem{MT13ar}
A.~Magyar, and T.~Titichetrakun, \emph{Corners in dense subsets of ${\mathbb P}^d$}, arXiv:1306.3026.


\bibitem{Tao06jam}
T.~Tao, \emph{The {G}aussian primes contain arbitrarily shaped constellations},
  J. Anal. Math. \textbf{99} (2006), 109--176.

\bibitem{TZ13ar}
T.~Tao and T.~Ziegler, \emph{A multi-dimensional {S}zemer\'edi theorem for the
  primes via a correspondence principle}, arXiv:1306.2886.

\bibitem{Va}
P.~Varnavides, \emph{On certain sets of positive density}, J. London Math. Soc. \textbf{34} (1959), 358--360.

\end{thebibliography}

\end{document}